\documentclass[11pt]{article}
\usepackage[margin=1.05in]{geometry}
\usepackage{amsmath,amssymb,amsthm}
\usepackage{booktabs}
\usepackage{tikz}
\usepackage[hidelinks]{hyperref}

\hypersetup{
  pdftitle={Balance Constants, Majority Cycles, and the Gold Partition
    Conjecture through Fourteen Elements},
  pdfauthor={Anish Gupta},
  pdfsubject={Linear extensions of finite partially ordered sets},
  pdfkeywords={posets, linear extensions, balance constant,
    Gold Partition Conjecture, 1/3-2/3 conjecture,
    linear extension majority cycle, computer-assisted proof}
}

\newtheorem{theorem}{Theorem}
\newtheorem{proposition}[theorem]{Proposition}
\newtheorem{lemma}[theorem]{Lemma}
\newtheorem{corollary}[theorem]{Corollary}
\theoremstyle{definition}
\newtheorem{observation}[theorem]{Observation}
\newtheorem{question}[theorem]{Question}
\newcommand{\ext}{\mathrm{e}}
\newcommand{\bal}{\mathrm{b}}
\newcommand{\Dig}{\mathrm{D}}

\newcommand{\NCYCLIC}{478{,}632{,}938}
\newcommand{\CTHREE}{477{,}954{,}774}
\newcommand{\CFOUR}{5{,}419{,}981}
\newcommand{\CFIVE}{33{,}473}
\newcommand{\CSIX}{10{,}423}
\newcommand{\CSEVEN}{61}
\newcommand{\CEIGHT}{30}
\newcommand{\NEQUALITY}{128}
\newcommand{\NTAILVALUES}{42}
\newcommand{\NTAILCLASSES}{469}
\newcommand{\NWITNESS}{44{,}013}
\newcommand{\MINCYCBAL}{2869/5810 = 0.4938\ldots}

\title{Balance Constants, Majority Cycles,\\
and the Gold Partition Conjecture\\
through Fourteen Elements}
\author{Anish Gupta\thanks{Independent researcher.
  \texttt{ag2269@cantab.ac.uk};
  ORCID: \href{https://orcid.org/0009-0008-8137-7729}
  {0009-0008-8137-7729}.}}
\date{July 2026}

\begin{document}
\maketitle

\begin{abstract}
We determine the exact extremal balance data of all
$1{,}338{,}193{,}159{,}771$ unlabeled posets on fourteen elements.  The least
balance constant exceeding $1/3$ is $37/106$.  The least over posets that are
not nontrivial ordinal sums is $254/725$, attained by a ladder with broken
rungs; this confirms a conjecture of Peczarski at order $14$, while orders
$12$ and $13$ reproduce De Loof, De Baets, and De Meyer.  No balance constant
lies in the gap above $1/3$ that Peczarski conjectures to be empty.  Exactly
$\NEQUALITY$ classes attain $1/3$, and every one is an ordinal sum of
singletons and copies of the three-element poset with one relation, a family
whose non-chain members are counted by
$a(n)-1$, where $a(n)=a(n-1)+a(n-3)$.  In the
linear-extension-majority digraph the
longest simple cycle has length $8$, against $7$ at order $13$, and exactly
$\CEIGHT$ classes attain it; of the thirteen such classes whose witnesses the
census retains, nine have a cycle spectrum containing no odd cycle at all.  A
second exhaustive pass over
the same classes verifies Peczarski's Gold Partition Conjecture through
fourteen elements, extending his order-$11$ frontier and implying in
particular that the $1/3$--$2/3$ Conjecture holds through order $14$.  All
arithmetic is exact and every extremal witness is recomputed by an independent
program.
\end{abstract}

\noindent\textbf{2020 Mathematics Subject Classification.}
Primary 06A07; Secondary 05A16, 68R05.

\noindent\textbf{Keywords.}
Partially ordered set, linear extension, balance constant, Gold Partition
Conjecture, $1/3$--$2/3$ Conjecture, linear extension majority cycle,
computer-assisted proof.

\section{Introduction}

Let $P$ be a finite poset on $n$ elements and write $\ext(P)$ for its number
of linear extensions.  For $x,y\in P$, let $P+xy$ denote the poset obtained by
adjoining $x<y$ and taking the transitive closure; if $y<x$ already, set
$\ext(P+xy)=0$.  Thus
\[
  \Pr[x \prec y] \;=\; \frac{\ext(P+xy)}{\ext(P)}
\]
is the probability that $x$ precedes $y$ in a uniformly chosen linear
extension.  For a non-chain $P$, its \emph{balance constant} is
\[
  \bal(P)=
  \max_{x\parallel y}\min\{\Pr[x \prec y],\,\Pr[y \prec x]\},
\]
the maximum over incomparable pairs.  The $1/3$--$2/3$ Conjecture asserts that
$\bal(P)\geq1/3$ for every finite non-chain.  It was independently posed by
Kislitsyn, Fredman, and Linial in connection with sorting under partial
information \cite{kislitsyn,fredman,linial}.  Kahn and Saks \cite{kahn-saks}
first proved a uniform positive bound, improved by Brightwell, Felsner, and
Trotter \cite{bft}; Chen \cite{chen} gives width-two constructions near the
conjectured gap, while Sah \cite{sah} gives the sharpest known quantitative
width-two bound.  Olson and Sagan \cite{olson-sagan} treat several structured
families, and Aires and Kahn
\cite{aires-kahn,aires-kahn-bounded} have recently obtained large-width and
bounded-width results.  Chan and Pak \cite{chan-pak} survey the area.

Two exhaustive computations over the order-$14$ classes are reported here.
The first is a census of exact balance and majority-cycle data
(Sections~\ref{sec:balance}--\ref{sec:cycles}); the second decides a stronger
conjecture of Peczarski for every class (Section~\ref{sec:certificates}).
They share the generator and the ideal-lattice machinery of
Section~\ref{sec:counting} but are separate passes with separate binaries,
and we keep their evidence separate.

Running both permits a comparison that neither pass records directly.  The
Gold Partition run counts certificates but not their posets; the census
records extremal posets but not their certificates.  Section~\ref{sec:meet}
classifies the retained extremal witnesses: many low-tail witnesses need a
triple certificate, whereas all $37$ retained length-$7$ or length-$8$
witnesses have a pair certificate.  Their deterministic retention rules are
biased, so this comparison does not estimate class-wide conditional rates.

\subsection{Extremal balance}

Peczarski \cite{peczarski-ladders} studied which posets are worst balanced.
He observed that the global minimum of $\bal$ above $1/3$ is not informative
on its own: forming the ordinal sum of $P$ with a singleton leaves $\bal$
unchanged, so the minimum is inherited from smaller orders rather than
attained afresh.  Restricting to posets that are not nontrivial ordinal sums
--- equivalently, not linear sums --- he conjectured that the worst is always
a \emph{ladder with broken rungs}, and separately that no poset has balance
constant strictly between $1/3$ and approximately $0.348843$.  He verified the first at orders $3$, $4$, $6$, $9$, $10$, and $11$ and
conjectured it for every order except $5$, $7$, and $8$, which he excludes
explicitly; at $n\geq12$ he restricted his searches to width two, width three,
and the ladder classes, judging a complete pass infeasible.

\begin{theorem}\label{thm:ladder}
The minimum balance constant over the order-$14$ posets that are not
nontrivial ordinal sums is $254/725$, attained by the ladder $L_{14,1,9}$.
\end{theorem}

De Loof, De Baets, and De Meyer determined the same quantity --- they call
such posets \emph{worst balanced} --- for every order up to $13$
\cite[\S4.2.2]{deloof-census}, and observed that orders $12$ and $13$ also
fall into Peczarski's class.  Theorem~\ref{thm:ladder} extends that by one
order; our orders $12$ and $13$ reproduce their result and are reported as a
regression, not as a contribution.

\begin{theorem}\label{thm:gap}
No poset on at most $14$ elements has balance constant strictly between $1/3$
and $0.348843$.  The least value above $1/3$ at order $14$ is
$37/106 = 0.349056\ldots$, inherited from a ten-element poset.
\end{theorem}

Here $0.348843$ is the decimal printed in \cite{peczarski-ladders} for a
conjectured limit, not an exact rational threshold, so Theorem~\ref{thm:gap}
verifies the conjecture at these orders rather than saying anything about the
limit itself.  A gap of this kind was asked for by Brightwell
\cite{brightwell} and is a theorem in width two: Sah \cite{sah} proves
$\bal(P)\geq(-3+5\sqrt{17})/52 = 0.33876\ldots$ for every width-two poset
that is not such an ordinal sum.  Theorem~\ref{thm:gap} is weaker in reach and
stronger in constant --- it holds over all widths but only to order $14$, and
it clears Peczarski's larger threshold rather than Sah's proved one.

\begin{proposition}\label{prop:equality}
Let $T$ be the three-element poset consisting of a two-element chain and an
isolated point.  Every ordinal sum of singletons and copies of $T$, other than
the chain, has balance constant exactly $1/3$, and the number of such posets
on $n$ elements is $a(n)-1$, where
\[
  a(n)=a(n-1)+a(n-3),\qquad a(0)=a(1)=a(2)=1.
\]
No other poset on at most $14$ elements has balance constant $1/3$.
\end{proposition}

The converse is a theorem in width two: Aigner \cite{aigner} showed that the
width-two posets with $\bal(P)=1/3$ are exactly the ordinal sums of singletons
and copies of $T$.  Every one of the $\NEQUALITY$ order-$14$ equality classes
has width two, so at order $14$ the census does not extend Aigner's
classification within width two; what it adds is that no poset of any larger
width attains $1/3$ either, which his argument does not cover.

For $n = 3,\ldots,14$ the counts are
$1, 2, 3, 5, 8, 12, 18, 27, 40, 59, 87, \NEQUALITY$.

\subsection{Majority cycles}

The \emph{linear-extension-majority} (LEM) digraph $\Dig(P)$ has vertex set
$P$ and an edge $x \to y$ whenever $\Pr[x\prec y] > 1/2$.  It has no loop and
no $2$-cycle, since $\Pr[x\prec y]+\Pr[y\prec x]=1$, but it need not be
acyclic: a directed cycle is a Condorcet-style failure of transitivity for the
majority relation.  The majority relation appeared in Kislitsyn
\cite{kislitsyn}; Fishburn exhibited cycles in 1974
\cite{fishburn-cycles} and later studied its majority graphs
\cite{fishburn-lem}.  Gehrlein and Fishburn made the first exhaustive
search for cycles, through order $9$ \cite{gehrlein-fishburn}.  De Loof, De
Baets, and De Meyer \cite{deloof-census} counted the posets carrying such a
cycle, by length, through order $13$, where the longest has length $7$ and is
attained by exactly one class.

\begin{theorem}\label{thm:cycle8}
The longest simple cycle in the linear-extension-majority digraph of an
order-$14$ poset has length $8$.  Exactly $\CEIGHT$ classes carry one.
\end{theorem}

\subsection{The Gold Partition Conjecture}

The Gold Partition Conjecture is stronger than the $1/3$--$2/3$ Conjecture.
It asks for two comparisons chosen in advance such that $t_0\geq t_1+t_2$ for
every sequence of outcomes, where $t_0$ is the initial number of linear
extensions and $t_i$ the number remaining after the first $i$ comparisons,
with $t_2=1$ when the first comparison leaves a chain.  Peczarski introduced
it, proved that it implies the $1/3$--$2/3$ Conjecture, and verified it
through $11$ elements \cite[Conjecture~1, Proposition~1, and
Theorem~3]{peczarski-gold}; the name refers to the resulting sorting bound
$C(P)\leq\log_{\varphi}\ext(P)$ with $\varphi=(1+\sqrt5)/2$
\cite[Propositions~2--3]{peczarski-gold}.  We are aware of no later exhaustive
extension.

\begin{theorem}\label{thm:main}
Every non-chain poset on at most $14$ elements satisfies the Gold Partition
Conjecture.
\end{theorem}

\begin{corollary}\label{thm:third}
Every non-chain poset on at most $14$ elements has balance constant at least
$1/3$.
\end{corollary}

Corollary~\ref{thm:third} follows from Theorem~\ref{thm:main} and Peczarski's
Proposition~1, and is confirmed independently by the census, which computes
$\bal(P)$ exactly for every class and finds no value below $1/3$.  Only its
order-$14$ case advances the previous computational frontier: De Loof, De
Baets, and De Meyer computed all mutual rank probabilities and determined the
worst balanced posets through order $13$~\cite{deloof-census}.

Structurally, Peczarski proved the conjecture for posets with a nontrivial
automorphism, for a class containing the $N$-free posets, and for $6$-thin
posets \cite{peczarski-comments,peczarski-thin}, and Dolores-Cuenca,
Guzm\'an-S\'aenz, and Kim showed it is preserved by lexicographic sum
\cite[Lemma~2.5]{dolores-cuenca}, an argument that applies unchanged to the
non-strict inequality.

\subsection{What is new}

Relative to the literature, the complete balance and majority-cycle census is
new at order $14$, while the Gold Partition verification is new at orders
$12$--$14$: Peczarski's exhaustive frontier was order $11$.  The worst
balanced posets through order $13$ reproduce De Loof et al.
\cite[\S4.2.2]{deloof-census}, as do the cycle counts at orders $9$ through
$13$, and the ladder values reproduce Peczarski's Table 1
\cite{peczarski-ladders}.  Relative to the first arXiv version of this paper,
the new material is the order-$14$ census and the synthesis of its data with
the Gold Partition certificates.  The all-pairs recurrence on the lattice of
order ideals is due to De Loof, De Meyer, and De Baets
\cite{deloof-ideals}; modulus sharding is standard, with prior art in both
enumeration papers we depend on \cite{brinkmann-mckay,deloof-census}.  Our
contribution is the new exhaustive ranges, their witnesses, the
certificate-directed implementation where its effect is measured, and the
reproducible verification package.

\section{Machinery}\label{sec:counting}

\subsection{Exact counting}

Let $\mathcal I(P)$ be the lattice of order ideals of $P$, and let $F(I)$ be
the number of linear extensions of the induced poset on $I$.  Then
\begin{equation}\label{eq:forward}
  F(\varnothing)=1,\qquad
  F(I)=\sum_{x\in\max(I)}F(I\setminus\{x\}).
\end{equation}
Assigning zero to ideals that violate an adjoined comparison gives
$\ext(P+xy)$ and $\ext(P+xy+yz)$.

If $B(I)$ is the number of ways to complete a linear extension whose initial
ideal is $I$, then one forward and one backward pass supply every pair count:
\begin{equation}\label{eq:pairs}
  \ext(P+xy)=
  \sum_{\substack{I\in\mathcal I(P)\\
                  x\in I,\ y\notin I\\
                  I\cup\{y\}\in\mathcal I(P)}}
        F(I)B(I\cup\{y\}).
\end{equation}
Each product in \eqref{eq:pairs} counts a subset of the linear extensions,
their sum is at most $\ext(P)$, and the comparisons use only small integer
multiples.  This ideal-lattice method for obtaining all mutual rank
probabilities in one forward and one backward pass is due to De Loof,
De Meyer, and De Baets \cite{deloof-ideals}.  Our contribution is not a new
all-pairs algorithm.

The two passes use it differently, and the difference is the reason there are
two.  The census needs every pair probability of every class, so it evaluates
\eqref{eq:pairs} in full for each class surviving the ordinal-sum and duality
reductions of Section~\ref{sec:balance}.  The Gold Partition decision needs
only a certificate, so it directs the same recurrences toward the cheapest
sufficient test and stops there (Section~\ref{sec:certificates}).

All extension counts are unsigned $64$-bit integers.  This is sufficient
through order $14$, since $\ext(P)\leq14!<2^{41}$; the formulas above use only
small integer multiples of these counts.  Every comparison of two
probabilities is an exact cross product in unsigned $128$-bit arithmetic, and
every balance constant is reduced and stored as a numerator--denominator pair.
No floating-point value is computed, compared, or stored anywhere in a
decision path; decimals appear only in printed reports.  The independent
checkers use Python integers and \texttt{fractions.Fraction} for the same
reason.

The posets are supplied by Brinkmann and McKay's \texttt{genposetg}, which
uses nauty for canonical labelling~\cite{brinkmann-mckay,mckay-piperno}.  Both
properties studied here are invariant under isomorphism, so one representative
of each unlabeled class suffices.  The generator's deterministic modulus filter
divides the stream into disjoint residues: the Gold Partition run used modulus
$4{,}096$ and the invocation
\begin{center}
\verb|gpc 14 o q m <residue> 4096|,
\end{center}
and the census used modulus $16{,}384$.  Every shard records the parameters it
was compiled with and the SHA-256 of the binary that produced it; each
aggregator refuses a shard set whose parameters or binary differ, verifies that
the residues are disjoint and covering, and checks the total against the
published class count.

Two capacities and one search budget could in principle be exceeded in the
census: the tail table, the equality table, and the per-poset cycle-search node
budget.  Each causes a nonzero exit, leaving the shard rejected rather than
silently truncated.  Nothing is capped quietly.

Among all posets attaining an extremal value the census retains the
lexicographically least encoding, so the reported witness does not depend on
the shard count or on scheduling.  One qualification: the duality filter
processes one member of each dual pair, so the witness is least among the
retained representatives.

\subsection{Ordinal sums}

Let the connected components of the incomparability graph of $P$ be
$P_1,\ldots,P_r$.  Any two elements in different components are comparable
and, for a fixed pair of components, all such comparisons point the same way,
so $P=P_1\oplus\cdots\oplus P_r$ is the finest ordinal-sum decomposition.  A
linear extension of an ordinal sum is exactly a concatenation of independently
chosen linear extensions of the summands, so the restriction map
$\mathcal E(P)\to\mathcal E(P_i)$ is surjective with all fibres of size
$\ext(P)/\ext(P_i)$.  Hence probabilities within a summand are computed in the
summand, and
\begin{equation}\label{eq:sumbal}
  \bal(P)=\max_i\bal(P_i),
\end{equation}
the maximum over non-singleton summands.  Cycles likewise cannot cross
summands in $\Dig_{\mathrm{inc}}$, and in $\Dig$ a crossing edge always points
the same way, so no cycle uses one.  The census decomposes and works summand
by summand, which is both a large saving and the reason the non-ordinal-sum
minimum is well defined.

Identity \eqref{eq:sumbal} has the consequence Peczarski noted, and it is easy
to overlook when reading a table of minima: $\bal(P\oplus\mathbf 1)=\bal(P)$,
so the global minimum above $1/3$ is non-increasing in $n$ and, at these
orders, is inherited from a smaller summand rather than newly attained by an
order-$n$ ordinal-indecomposable poset.  The value $37/106$ comes from
$L_{10,1,5}$, a \emph{ten}-element ladder, padded with singletons.  The quantity
that can still move with $n$, and the one
Theorem~\ref{thm:ladder} is about, is the minimum over posets that are not
nontrivial ordinal sums.

\section{Balance constants}\label{sec:balance}

\subsection{The low tail}

Figure~\ref{fig:tail} shows every distinct balance constant in the window
$(1/3,\,9/25]$, with the number of classes attaining it.  The window is
compiled into the binary and comfortably contains every published record.

\begin{figure}[ht]
\centering
\begin{tikzpicture}[x=1cm, y=1cm,
  axis/.style={draw=black!75, line width=0.4pt},
  nonsum/.style={draw=black, line width=0.9pt},
  sum/.style={draw=black!45, line width=0.6pt},
  lead/.style={draw=black!35, line width=0.3pt},
]
  \fill[black!8] (0,0) rectangle (6.979,2.70);
  \draw[axis, densely dashed] (6.979,0) -- (6.979,2.70);
  \node[font=\scriptsize, anchor=south east, black!55] at (6.879,2.42) {conjectured gap: no class here};
  \draw[axis] (0,0) -- (12.00,0);
  \node[font=\scriptsize, anchor=north] at (0,-0.08) {$\tfrac13$};
  \node[font=\scriptsize, anchor=north] at (12.00,-0.08) {$\tfrac9{25}$};
  \draw[sum] (7.075,0) -- (7.075,0.576);
  \draw[lead] (7.075,0.576) -- (7.045,1.02);
  \node[font=\tiny, anchor=west, inner sep=1pt, fill=white, fill opacity=.9, text opacity=1] at (7.075,1.02) {$\frac{37}{106}$\,{\scriptsize$\times$}11};
  \draw[sum] (7.581,0) -- (7.581,0.401);
  \draw[lead] (7.581,0.401) -- (7.045,1.62);
  \node[font=\tiny, anchor=west, inner sep=1pt, fill=white, fill opacity=.9, text opacity=1] at (7.075,1.62) {$\frac{97}{277}$\,{\scriptsize$\times$}3};
  \draw[nonsum] (7.655,0) -- (7.655,0.290);
  \draw[lead] (7.655,0.290) -- (7.045,2.22);
  \node[font=\tiny, anchor=west, inner sep=1pt, fill=white, fill opacity=.9, text opacity=1] at (7.075,2.22) {$\frac{254}{725}$$^{\ast}$\,{\scriptsize$\times$}1};
  \draw[sum] (7.701,0) -- (7.701,0.355);
  \draw[lead] (7.701,0.355) -- (9.295,1.02);
  \node[font=\tiny, anchor=west, inner sep=1pt, fill=white, fill opacity=.9, text opacity=1] at (9.325,1.02) {$\frac{157}{448}$\,{\scriptsize$\times$}2};
  \draw[sum] (7.895,0) -- (7.895,0.490);
  \draw[lead] (7.895,0.490) -- (9.295,1.62);
  \node[font=\tiny, anchor=west, inner sep=1pt, fill=white, fill opacity=.9, text opacity=1] at (9.325,1.62) {$\frac{20}{57}$\,{\scriptsize$\times$}6};
  \draw[nonsum] (8.733,0) -- (8.733,0.355);
  \draw[lead] (8.733,0.355) -- (9.295,2.22);
  \node[font=\tiny, anchor=west, inner sep=1pt, fill=white, fill opacity=.9, text opacity=1] at (9.325,2.22) {$\frac{103}{292}$$^{\ast}$\,{\scriptsize$\times$}2};
  \draw[nonsum] (8.745,0) -- (8.745,0.355);
  \draw[sum] (8.824,0) -- (8.824,0.722);
  \draw[sum] (8.922,0) -- (8.922,0.355);
  \draw[sum] (9.231,0) -- (9.231,0.649);
  \draw[sum] (9.437,0) -- (9.437,0.401);
  \draw[nonsum] (9.623,0) -- (9.623,0.355);
  \draw[sum] (9.736,0) -- (9.736,0.436);
  \draw[sum] (10.000,0) -- (10.000,0.727);
  \draw[sum] (10.130,0) -- (10.130,0.490);
  \draw[sum] (10.274,0) -- (10.274,0.576);
  \draw[sum] (10.443,0) -- (10.443,0.401);
  \draw[nonsum] (10.526,0) -- (10.526,0.290);
  \draw[sum] (10.714,0) -- (10.714,0.900);
  \draw[sum] (10.766,0) -- (10.766,0.436);
  \draw[nonsum] (10.794,0) -- (10.794,0.290);
  \draw[sum] (10.800,0) -- (10.800,0.490);
  \draw[sum] (10.830,0) -- (10.830,0.490);
  \draw[sum] (10.837,0) -- (10.837,0.490);
  \draw[nonsum] (11.057,0) -- (11.057,0.355);
  \draw[sum] (11.111,0) -- (11.111,0.649);
  \draw[sum] (11.290,0) -- (11.290,0.490);
  \draw[nonsum] (11.381,0) -- (11.381,0.355);
  \draw[nonsum] (11.435,0) -- (11.435,0.290);
  \draw[nonsum] (11.437,0) -- (11.437,0.290);
  \draw[sum] (11.538,0) -- (11.538,0.827);
  \draw[sum] (11.569,0) -- (11.569,0.490);
  \draw[sum] (11.638,0) -- (11.638,0.436);
  \draw[nonsum] (11.693,0) -- (11.693,0.290);
  \draw[nonsum] (11.722,0) -- (11.722,0.290);
  \draw[sum] (11.818,0) -- (11.818,0.490);
  \draw[sum] (11.826,0) -- (11.826,0.355);
  \draw[sum] (11.842,0) -- (11.842,0.490);
  \draw[sum] (11.871,0) -- (11.871,0.727);
  \draw[sum] (11.921,0) -- (11.921,0.355);
  \draw[sum] (11.960,0) -- (11.960,0.436);
  \draw[sum] (12.000,0) -- (12.000,0.859);
\end{tikzpicture}
\caption{The low balance tail: all $\NTAILVALUES$ distinct balance constants
of order-$14$ posets in $(1/3,\,9/25]$.  The shaded region is the gap
conjectured empty in \cite{peczarski-ladders}; no class lies in it.  Stem
height grows with the number of classes attaining the value, shown as
$\times k$; values marked $\ast$ are attained by a poset that is not a
nontrivial ordinal sum.  Six values are labelled; the rest are drawn
unlabelled.}
\label{fig:tail}
\end{figure}
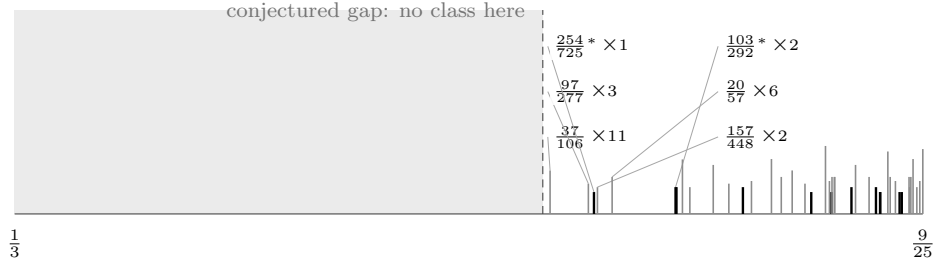

At order $14$ there are $\NTAILVALUES$ distinct values in the window, attained
by $\NTAILCLASSES$ classes in all.  Exactly $30$ are the order-$13$ tail
inherited by ordinal-sum padding; the remaining $12$ are attained by an
order-$14$ non-sum.  The values cluster immediately above the conjectured gap
rather than spreading across the window: the five smallest,
$37/106$, $97/277$, $254/725$, $157/448$, and $20/57$, all lie within
$2\times10^{-3}$ of each other, and the whole window holds fewer than five
hundred classes out of $1.3$ trillion.  The $12$ new values are attained by
one or two non-sum classes each; their minimum is the quantity
Theorem~\ref{thm:ladder} concerns.

\subsection{Ladders with broken rungs}

For $n\geq3$ the ladder $L_n$ has elements $u_1<u_2<\cdots$ and
$v_1<v_2<\cdots$ interleaved by rungs; breaking a rung deletes one cross
relation.  Peczarski writes $L_{n,i_1,\ldots,i_k}$ for the $n$-element ladder
with rungs $i_1,\ldots,i_k$ broken \cite{peczarski-ladders}.

We did not quote his table.  For each order the program constructs every
broken-rung ladder from his definition, computes its balance constant exactly,
and takes the minimum; at order $14$ this is an exhaustive search over all
$2^{11}=2{,}048$ rung subsets.  The result is then compared with the census
minimum over non-ordinal-sums, which is what proves Theorem~\ref{thm:ladder}.

\begin{table}[ht]
\centering
\begin{tabular}{rlll}
\toprule
$n$ & census minimum, non-sum & worst broken-rung ladder & agree\\
\midrule
 7 & $14/39$   & $9/25 = L_{7,1,2}$      & no$^\dagger$\\
 8 & $16/45$   & $17/46 = L_{8,1,2,3}$   & no$^\dagger$\\
 9 & $6/17$    & $6/17 = L_{9,1,2,3,4}$  & yes\\
10 & $37/106$  & $37/106 = L_{10,1,5}$   & yes\\
11 & $20/57$   & $20/57 = L_{11,1,6}$    & yes\\
12 & $97/277$  & $97/277 = L_{12,1,7}$   & yes$^\ddagger$\\
13 & $157/448$ & $157/448 = L_{13,1,8}$  & yes$^\ddagger$\\
14 & $254/725$ & $254/725 = L_{14,1,9}$  & yes, new\\
\bottomrule
\end{tabular}
\caption{Peczarski's broken-rung conjecture against the census.  Each ladder
value was recomputed by exhaustive search over rung subsets, not quoted.
$\dagger$: $n=5,7,8$ are the exceptions \cite{peczarski-ladders} excludes
explicitly; his conjecture is stated for every other order, and verified there
at $3$, $4$, $6$, $9$, $10$, and $11$.  $\ddagger$: also obtained by De Loof
et al.\ \cite[\S4.2.2]{deloof-census}; our rows reproduce theirs.}
\label{tab:ladder}
\end{table}

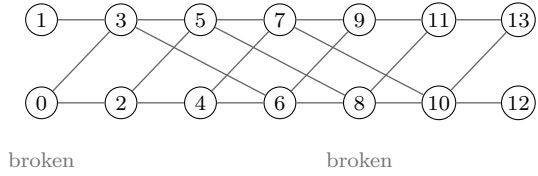
\begin{figure}[ht]
\centering
\begin{tikzpicture}[
  x=1cm, y=1cm,
  el/.style={circle, draw, minimum size=4.2mm, inner sep=1pt,
             font=\scriptsize},
  rail/.style={draw=black!60, line width=0.5pt},
  rung/.style={draw=black!60, line width=0.5pt},
]
  \node[el] (l0) at (0.00,0.00) {0};
  \node[el] (l1) at (0.00,1.10) {1};
  \node[el] (l2) at (1.05,0.00) {2};
  \node[el] (l3) at (1.05,1.10) {3};
  \node[el] (l4) at (2.10,0.00) {4};
  \node[el] (l5) at (2.10,1.10) {5};
  \node[el] (l6) at (3.15,0.00) {6};
  \node[el] (l7) at (3.15,1.10) {7};
  \node[el] (l8) at (4.20,0.00) {8};
  \node[el] (l9) at (4.20,1.10) {9};
  \node[el] (l10) at (5.25,0.00) {10};
  \node[el] (l11) at (5.25,1.10) {11};
  \node[el] (l12) at (6.30,0.00) {12};
  \node[el] (l13) at (6.30,1.10) {13};
  \draw[rail] (l0) -- (l2);
  \draw[rung] (l0) -- (l3);
  \draw[rail] (l1) -- (l3);
  \draw[rail] (l2) -- (l4);
  \draw[rung] (l2) -- (l5);
  \draw[rail] (l3) -- (l5);
  \draw[rung] (l3) -- (l6);
  \draw[rail] (l4) -- (l6);
  \draw[rung] (l4) -- (l7);
  \draw[rail] (l5) -- (l7);
  \draw[rung] (l5) -- (l8);
  \draw[rail] (l6) -- (l8);
  \draw[rung] (l6) -- (l9);
  \draw[rail] (l7) -- (l9);
  \draw[rung] (l7) -- (l10);
  \draw[rail] (l8) -- (l10);
  \draw[rung] (l8) -- (l11);
  \draw[rail] (l9) -- (l11);
  \draw[rail] (l10) -- (l12);
  \draw[rung] (l10) -- (l13);
  \draw[rail] (l11) -- (l13);
  \node[font=\scriptsize, black!55] at (0.00,-0.75) {broken};
  \node[font=\scriptsize, black!55] at (4.20,-0.75) {broken};
\end{tikzpicture}
\caption{$L_{14,1,9}$, the worst broken-rung ladder at order $14$, with
$\bal=254/725$ and $\ext=725$.  The two rungs marked \emph{broken} are the
missing cross relations.}
\label{fig:ladder}
\end{figure}

Table~\ref{tab:ladder} agrees with his Table 1 at every order where the two
overlap, including the two exceptions he records.

A reader tabulating $n=10$ through $14$ will notice that the extremal ladder
is $L_{n,1,n-5}$ throughout, and should know that this does not continue: it
fails at $n=9$, where four rungs break, and Peczarski's Table 1 gives
$L_{15,1,5,6,10}$ and $L_{17,1,5,8,12}$.

\subsection{The equality locus}

\begin{proof}[Proof of Proposition~\ref{prop:equality}]
The poset $T$ has $\ext(T)=3$ and $\bal(T)=1/3$.  By \eqref{eq:sumbal}, an
ordinal sum of singletons and copies of $T$ has balance constant $1/3$ as soon
as at least one summand is a copy of $T$; if all summands are singletons the
sum is a chain and has no incomparable pair.  Decomposition into
ordinal-indecomposable summands is unique, so these posets are in bijection
with the compositions of $n$ into parts $1$ and $3$.  Splitting on the first
part gives $a(n)=a(n-1)+a(n-3)$, and discarding the all-singleton composition
gives $a(n)-1$.
\end{proof}

The converse --- that nothing else attains $1/3$ --- is a computational
statement, verified here at every order through $14$ and, to our knowledge,
not proved.  The aggregator does not merely count: it recomputes each
representative's balance constant and separately recognises the family
structurally, by decomposing the poset and checking that every summand is a
singleton or a copy of $T$, without computing any probability.  The two
recognitions must agree.

\section{Majority cycles}\label{sec:cycles}

\subsection{Which majority relation}

De Loof et al.\ define the LEM relation as the strict $1/2$-cut of the mutual
rank probability relation over \emph{all} ordered pairs of distinct elements.
A comparable pair $x<y$ has $\Pr[x\prec y]=1$ and so contributes an edge.
Writing $\Dig_{\mathrm{inc}}(P)$ for the sub-digraph on incomparable pairs
only, $\Dig(P)$ is $\Dig_{\mathrm{inc}}(P)$ together with the strict order
relation of $P$.

One would like to work with the cheaper $\Dig_{\mathrm{inc}}(P)$, arguing that
comparable edges cannot lie on a cycle.  That is provable only for triangles.

\begin{proposition}\label{prop:three}
No $3$-cycle of $\Dig(P)$ uses a comparable edge.
\end{proposition}

\begin{proof}
Suppose $x\to y\to z\to x$ with $x<y$ in $P$.  Then $z\to x$ and $x<y$ give
$\Pr[z\prec y]\geq\Pr[z\prec x]>1/2$, since every extension with $z\prec x$
also has $z\prec y$.  Thus $z\to y$, contradicting $y\to z$.
\end{proof}

The natural extension of this argument replaces a $k$-cycle through a
comparable edge by a $(k-1)$-cycle which need not contain one, so it yields a
contradiction only at $k=3$.  Rather than carry an unproved lemma into a
computation of this size, the census builds the full relation $\Dig(P)$, which
is what \cite{deloof-census} defines, and computes the restricted spectrum
alongside it as a separate set of counters.  The comparison then becomes
measured data.  Adding the comparable edges costs nothing: they are read off
the order relation and require no linear-extension counting.

\subsection{The spectrum at order 14}

\begin{table}[ht]
\centering
\begin{tabular}{rrrrrrrr}
\toprule
$n$ & any & $3$ & $4$ & $5$ & $6$ & $7$ & $8$\\
\midrule
 9 & 5           & 5           & 0        & 0    & 0   & 0        & 0\\
10 & 153         & 148         & 6        & 0    & 0   & 0        & 0\\
11 & 5{,}815     & 5{,}740     & 101      & 0    & 0   & 0        & 0\\
12 & 218{,}097   & 216{,}573   & 2{,}885  & 5    & 21  & 0        & 0\\
13 & 9{,}348{,}400 & 9{,}318{,}881 & 102{,}127 & 471 & 363 & 1     & 0\\
14 & \NCYCLIC    & \CTHREE     & \CFOUR   & \CFIVE & \CSIX & \CSEVEN & \CEIGHT\\
\bottomrule
\end{tabular}
\caption{Classes carrying a simple majority cycle of each length.  A class may
carry cycles of several lengths, so the ``any'' column is not the row sum.
Orders $9$--$13$ reproduce \cite{deloof-census}.  No class at order $14$
carries a cycle of length $9$ or more.}
\label{tab:cycles}
\end{table}

The proportion of classes carrying a cycle continues to rise, as Gehrlein and
Fishburn conjectured \cite{gehrlein-fishburn}: at order $14$ it is $3.5767$
per $10^4$, against $2.764$ at order $13$.  De Loof et al.\ fitted
$f(n)=a+bn+n\log n$ to their orders $9$ through $13$, obtaining $a=5.2304$
and $b=-2.7572$, and expected it to approximate the incidence for $n$ below
about $50$.  Their fit gives $f(14)=3.5764$, against the measured $3.5767$.
The residual is $0.00031$, but one out-of-sample point cannot validate the
functional form.

The counts fall steeply in length but not monotonically.  At order $12$ there
are $21$ classes with a $6$-cycle and only $5$ with a $5$-cycle; at orders
$13$ and $14$ the order reverses, though the $5$-to-$6$ ratio stays far below
the $3$-to-$4$ ratio.  Section~\ref{sec:even} offers a structural reason why
even lengths can be over-represented.

\subsection{Cycles of length eight}

Exactly $\CEIGHT$ classes carry an $8$-cycle.  The census reports as its
canonical witness the lexicographically least encoding,
\begin{center}
\texttt{14:0000000000000000000300050006000300050002000400f7035705a7},
\end{center}
in the encoding of Section~\ref{sec:verify}, with $\ext(P)=7{,}511{,}140$,
$\bal(P)=1/2$, width $8$, height $3$, and cycle spectrum $\{3,4,5,6,7,8\}$.

Figure~\ref{fig:cycle8} draws a different member of the family,
\begin{center}
\texttt{14:0000000000000003000500030005000300050002000401ff043f03c7},
\end{center}
chosen because its cycle spectrum is exactly $\{4,8\}$ --- it carries an
$8$-cycle and no odd cycle at all --- which is the phenomenon of
Section~\ref{sec:even}.  It has three minimal elements, an eight-element
antichain, and three maximal elements.  For both posets the restricted digraph
$\Dig_{\mathrm{inc}}(P)$ has the same spectrum as $\Dig(P)$.

\begin{figure}[ht]
\centering
\begin{tikzpicture}[
  x=1cm, y=1cm,
  every node/.style={inner sep=1pt},
  el/.style={circle, draw=black!65, minimum size=4.4mm,
              font=\scriptsize},
  cyc/.style={el, draw=black, line width=0.9pt},
  odd/.style={el, fill=black!12},
  cov/.style={draw=black!55, line width=0.35pt},
  maj/.style={->, >=stealth, draw=black!30, line width=0.4pt},
  hot/.style={->, >=stealth, draw=black, line width=1.0pt},
]
  \node[el] (p11) at (-2.080,0.000) {11};
  \node[el] (p12) at (0.000,0.000) {12};
  \node[el] (p13) at (2.080,0.000) {13};
  \node[cyc] (p3) at (-2.730,1.550) {3};
  \node[cyc] (p5) at (-1.950,1.550) {5};
  \node[cyc] (p7) at (-1.170,1.550) {7};
  \node[cyc] (p4) at (-0.390,1.550) {4};
  \node[cyc] (p6) at (0.390,1.550) {6};
  \node[cyc] (p8) at (1.170,1.550) {8};
  \node[cyc] (p9) at (1.950,1.550) {9};
  \node[cyc] (p10) at (2.730,1.550) {10};
  \node[el] (p0) at (-2.080,3.100) {0};
  \node[el] (p1) at (0.000,3.100) {1};
  \node[el] (p2) at (2.080,3.100) {2};
  \draw[cov] (p3) -- (p0);
  \draw[cov] (p3) -- (p1);
  \draw[cov] (p4) -- (p0);
  \draw[cov] (p4) -- (p2);
  \draw[cov] (p5) -- (p0);
  \draw[cov] (p5) -- (p1);
  \draw[cov] (p6) -- (p0);
  \draw[cov] (p6) -- (p2);
  \draw[cov] (p7) -- (p0);
  \draw[cov] (p7) -- (p1);
  \draw[cov] (p8) -- (p0);
  \draw[cov] (p8) -- (p2);
  \draw[cov] (p9) -- (p1);
  \draw[cov] (p10) -- (p2);
  \draw[cov] (p11) -- (p3);
  \draw[cov] (p11) -- (p4);
  \draw[cov] (p11) -- (p5);
  \draw[cov] (p11) -- (p6);
  \draw[cov] (p11) -- (p7);
  \draw[cov] (p11) -- (p8);
  \draw[cov] (p12) -- (p3);
  \draw[cov] (p12) -- (p4);
  \draw[cov] (p12) -- (p5);
  \draw[cov] (p12) -- (p10);
  \draw[cov] (p13) -- (p6);
  \draw[cov] (p13) -- (p7);
  \draw[cov] (p13) -- (p8);
  \draw[cov] (p13) -- (p9);
  \node[font=\footnotesize] at (0,-0.95) {(a) $P$, with $\mathrm{e}(P)=4,369,616$};
  \node[odd] (m3) at (6.420,3.300) {3};
  \node[el] (m4) at (7.657,2.787) {4};
  \node[odd] (m6) at (8.170,1.550) {6};
  \node[el] (m7) at (7.657,0.313) {7};
  \node[odd] (m5) at (6.420,-0.200) {5};
  \node[el] (m9) at (5.183,0.313) {9};
  \node[odd] (m8) at (4.670,1.550) {8};
  \node[el] (m10) at (5.183,2.787) {10};
  \draw[hot] (m3) to (m4);
  \draw[maj] (m3) to[bend left=12] (m9);
  \draw[hot] (m4) to (m6);
  \draw[maj] (m4) to[bend left=12] (m8);
  \draw[maj] (m5) to[bend left=12] (m4);
  \draw[hot] (m5) to (m9);
  \draw[hot] (m6) to (m7);
  \draw[maj] (m6) to[bend left=12] (m10);
  \draw[maj] (m7) to[bend left=12] (m3);
  \draw[hot] (m7) to (m5);
  \draw[maj] (m8) to[bend left=12] (m7);
  \draw[hot] (m8) to (m10);
  \draw[maj] (m9) to[bend left=12] (m6);
  \draw[hot] (m9) to (m8);
  \draw[hot] (m10) to (m3);
  \draw[maj] (m10) to[bend left=12] (m5);
  \node[font=\footnotesize] at (6.420,-0.95) {(b) $D(P)$ restricted to the cyclic part};
\end{tikzpicture}
\caption{An order-$14$ poset $P$ whose majority digraph contains an $8$-cycle.
In (a) the elements lying on a cycle are drawn with a heavy border; they are
exactly the eight-element middle antichain.  In (b) the highlighted $8$-cycle
is $3\to4\to6\to7\to5\to9\to8\to10\to3$, and the two shadings are the two
sides of a bipartition of the underlying graph, so every cycle of $\Dig(P)$
has even length.  Here $\ext(P)=4{,}369{,}616$ and $\bal(P)=1/2$ exactly.}
\label{fig:cycle8}
\end{figure}

\subsection{Even spectra}\label{sec:even}

Call the \emph{cyclic part} of $\Dig(P)$ the set of vertices lying on at least
one directed cycle.  The following is immediate but explains a pattern in the
data.

\begin{observation}\label{obs:even}
If the underlying undirected graph of the cyclic part of $\Dig(P)$ is
bipartite, then every simple cycle of $\Dig(P)$ has even length.  In
particular $\Dig(P)$ has no $3$-cycle.
\end{observation}

The census retains one witness per cycle length per shard, giving $13$
distinct posets carrying an $8$-cycle out of the $\CEIGHT$ that exist, and
$28$ carrying a $7$-cycle.  Of the $13$, \emph{nine} have
bipartite cyclic part and spectrum exactly $\{4,8\}$; the remaining four have
the full spectrum $\{3,4,5,6,7,8\}$.  None of the $28$ length-$7$ witnesses is
bipartite, as Observation~\ref{obs:even} requires, since a $7$-cycle is odd.

The implication does not reverse.  Of the $3{,}212$ retained witnesses with
an all-even spectrum, seven have nonbipartite cyclic part; their spectrum is
exactly $\{4\}$ in every case.  One is
\texttt{14:00000000000000000000000000070019002a0034000100200413082c}.  An
undirected odd cycle need not be consistently oriented, so it need not carry a
directed one.

Although $3$-cycles outnumber $8$-cycles by seven orders of magnitude
(Table~\ref{tab:cycles}), the witness in Figure~\ref{fig:cycle8} has an
$8$-cycle and no $3$-cycle.  In that witness the
cyclic part is a $2$-in, $2$-out bipartite digraph on $4+4$ vertices, and the
$8$-cycle is a Hamiltonian cycle of that digraph.  We do not know whether the
even-spectrum posets form a natural class, nor whether the $\{4,8\}$ spectrum
is forced once the cyclic part is bipartite and has eight vertices.

Two further regularities are recorded as observations because we can neither
prove them nor rule out that they are artefacts of how witnesses are sampled:
within each shard the census retains the lexicographically least encoding for
each cycle length, so the witness set is determined by the encoding rather
than drawn uniformly from the classes.

\begin{observation}\label{obs:half}
All $37$ recorded order-$14$ posets carrying a cycle of length $7$ or $8$ have
$\bal(P)=1/2$ exactly.
\end{observation}

Since $\bal(P)\leq1/2$ always, this says the poset is perfectly balanced.  The
natural reading --- that a long majority cycle forces near-ties --- should be
resisted without a control, and the control weakens it.  Across all
$\NWITNESS$ distinct witnesses the census retains, the proportion with
$\bal(P)=1/2$ is $71.0\%$ at length $3$ ($11{,}626$ of $16{,}384$), $49.8\%$
at length $4$ ($8{,}162$ of $16{,}384$), $27.2\%$ at length $5$ ($2{,}296$ of
$8{,}428$), and $98.8\%$ at length $6$ ($3{,}067$ of $3{,}103$).  Exact
balance is therefore common among
cycle-carrying posets generally, and conspicuously not monotone in cycle
length, so the value $100\%$ at lengths $7$ and $8$ --- over $28$ and $13$
witnesses --- is suggestive rather than evidence of a universal mechanism.
A transposing automorphism of an incomparable pair forces that pair to be
exactly balanced \cite[Prop.~2.3]{olson-sagan}, and $35$ of the $37$ witnesses
admit one.  The remaining two have trivial automorphism group and are exactly
balanced anyway.

\begin{observation}\label{obs:antichain}
The cyclic part of $\Dig(P)$ is an antichain in $P$ for every one of the $37$
recorded posets carrying a cycle of length $7$ or $8$, but not in general:
of the $\NWITNESS$ distinct retained witnesses, $166$ have a cyclic part
containing a comparable pair.  They are concentrated at length $6$, where
$118$ of $3{,}103$ do, against under $0.42\%$ at every other length.
\end{observation}

Thus Proposition~\ref{prop:three} does not extend to the whole cyclic part:
$166$ retained witnesses contain a comparable pair among their cyclic
vertices, although no directed $3$-cycle uses a comparable edge.  This is
compatible with Observation~\ref{obs:relations}: a comparable pair may lie in
the cyclic part without any cycle using the edge between them.

For every order-$14$ class the census computes the cycle spectrum of $\Dig(P)$
and of $\Dig_{\mathrm{inc}}(P)$ separately.

\begin{observation}\label{obs:relations}
For every enumerated poset through order $14$, the full and restricted
majority digraphs have the same cycle spectrum.
\end{observation}

This is an exhaustive computational conclusion, not merely a comparison of
retained witnesses.  For each length, the classes carrying that cycle in
$\Dig_{\mathrm{inc}}(P)$ form a subset of those carrying it in $\Dig(P)$,
because the former is a subdigraph of the latter.  The exhaustive counters
give equal cardinalities for these two sets at every tested order and length,
so the sets, and hence every class's spectrum, agree.  Proposition
\ref{prop:three} proves the length-$3$ case in general; the longer-cycle
statement remains Question~\ref{q:relations} below.

\section{The Gold Partition Conjecture}\label{sec:certificates}

\subsection{Certificates}

For an ordered incomparable pair $(x,y)$, Peczarski calls $z$ a \emph{slave}
if either $z>x$ and $z\parallel y$, or $z<y$ and $z\parallel x$.  The verifier
uses the following sufficient conditions from Peczarski's Definitions 1--2 and
Lemmas 1--2 \cite[pp.~91--92]{peczarski-gold}.

\begin{lemma}[Peczarski]\label{lem:certificates}
A non-chain $P$ has the Gold Partition property if one of the following
conditions holds.
\begin{enumerate}
  \item There is an ordered incomparable pair $(x,y)$ with at most one slave
  and
  \[
    2\ext(P+xy)\geq\ext(P).
  \]
  \item There are three distinct elements $x,y,z$ such that
  \[
    \ext(P+xy+yz)
      \leq \max\{\ext(P+yx),\ext(P+zy)\}
      \leq \frac{\ext(P)}2.
  \]
\end{enumerate}
\end{lemma}

The elements in condition (2) need not be pairwise incomparable.  Peczarski's
separate cyclic-triple screen is subsumed by the exhaustive search for
condition (2).  In each case the proof supplies two comparisons fixed before
their outcomes are known, as required by the conjecture.

The verifier also accepts a half-balanced incomparable pair $(x,y)$, meaning
\[
  2\ext(P+xy)=2\ext(P+yx)=\ext(P).
\]
Indeed, use this pair as the first comparison.  Then $t_1=t_0/2$, while any
fixed second comparison leaves at most $t_1$ extensions; hence
$t_0\geq t_1+t_2$.  This direct argument is also recorded by
Peczarski~\cite{peczarski-comments}.

There is a useful inexpensive screen for condition (1).  If both orientations
of an incomparable pair have at most one slave, one of $\ext(P+xy)$ and
$\ext(P+yx)$ is at least $\ext(P)/2$, so condition (1) holds without counting
either quantity.  The verifier applies this bilateral screen first.  It then
searches for a low-slave pair or a half-balanced pair, and finally for a
triple satisfying condition (2).  The division between the first two
certificate counters depends on search order; their union does not, and only
the union is reported.

Candidate pairs are tried in a simple structural order and candidate triples
by decreasing right-hand bound in Lemma~\ref{lem:certificates}.  In the
production classifier two configurations need no new recurrence:
$\ext(P+xy+yz)=0$ when $z<x$, and when $x<z$ the positions of $y$ relative to
$x<z$ partition the extensions, giving
$\ext(P)-\ext(P+yx)-\ext(P+zy)$.

\subsection{The order-14 decision}\label{sec:gpcresult}

\begin{table}[ht]
\centering
\begin{tabular}{lrrr}
\toprule
 & $n=12$ & $n=13$ & $n=14$\\
\midrule
isomorphism classes
 & 1{,}104{,}891{,}746
 & 33{,}823{,}827{,}452
 & 1{,}338{,}193{,}159{,}771\\
chain
 & 1 & 1 & 1\\
pair certificate
 & 1{,}066{,}006{,}204
 & 32{,}418{,}324{,}910
 & 1{,}272{,}077{,}147{,}789\\
balanced-triple certificate
 & 38{,}885{,}541
 & 1{,}405{,}502{,}541
 & 66{,}116{,}011{,}981\\
open
 & 0 & 0 & 0\\
\bottomrule
\end{tabular}
\caption{Gold Partition certificate aggregates under a pair-first
classification.  ``Pair certificate'' means that a low-slave or half-balanced
pair exists; ``balanced-triple certificate'' means that no pair certificate
exists and a triple does.  For each order the four rows partition the total.}
\label{tab:gpc}
\end{table}

The order-$14$ aggregate contains all $4{,}096$ residues and agrees with the
published number of unlabeled posets, first obtained by Heitzig and Reinhold
and independently confirmed by Brinkmann and McKay
\cite{heitzig-reinhold,brinkmann-mckay}.  Exactly one class is a chain, the
certificate counts partition the total, and the open count is zero, which is
Theorem~\ref{thm:main}.  The known total checks enumeration completeness; the
partition and shard checks are internal consistency tests; the zero open count
is the computational conclusion.  The released aggregator was rerun on the
preserved $4{,}096$-shard archive, and the same executable reproduces the
known class totals at orders $10$ through $13$ together with the certificate
counts of an earlier run.

The share of classes needing the weaker balanced-triple certificate rises
steadily: $2.65$, $3.02$, $3.52$, $4.16$, and $4.94$ percent at orders $10$
through $14$.  Peczarski observes that his Table I might suggest almost all
posets fall to the low-slave condition and shows this is false: by a zero-one
law the fraction of posets satisfying its hypothesis tends to $0$
\cite[p.~94]{peczarski-gold}.  The five measured orders are consistent with
that result, but the reported pair category also includes half-balanced pairs,
so it does not by itself imply that the balanced-triple share must rise.  How
that share behaves, and whether the certificate scheme stays practical at
orders we cannot reach, the data does not say.

The order-$14$ calculation took $1{,}694$ user core-hours over a $19$-hour
shard span on AMD EPYC 7B13 processors; the census took $5{,}103$ user
core-hours over a $53$-hour shard span on the same hardware, the difference
being that a census cannot stop at a certificate.  In a controlled ablation on
the production platform, removing only the bilateral early exit increased
median user time by factors of $4.25$, $4.81$, $5.57$, and $8.55$ at orders
$10$ through $13$, respectively.  The first figure is a complete pass; the others
aggregate the same three fixed modulus shards over alternating repetitions.
The baseline binary is the production binary, and the total, pair-certificate
union, balanced-triple count, and open count agreed in every comparison.  This
measures the screening step within the present implementation, not a speedup
over the programs of Peczarski or De Loof et al.

\section{Where the extremal posets sit}\label{sec:meet}

Sections~\ref{sec:balance} and~\ref{sec:cycles} produce two families of
extremal objects; Section~\ref{sec:certificates} partitions every class by the
certificate that resolves it.  Neither program records what the other
measures, so we classified the recorded extremal witnesses with an independent
implementation of Lemma~\ref{lem:certificates} (Section~\ref{sec:verify}).

These are the witness sets the census retains --- one per distinct extremal
value, one per cycle length per shard --- not samples drawn from the classes,
and they cannot estimate class-wide rates.  How much that matters is visible
directly: of the $42$ one-per-value low-tail witnesses $21$ have no pair
certificate, while of the $180$ distinct low-tail witnesses retained across
all shards $64$ also have no pair certificate.  The same property, two
retention rules, $50\%$ against $35.6\%$.  For scale only, $4.94\%$ of all
order-$14$ classes need the balanced triple (Table~\ref{tab:gpc}); that
exhaustive class rate is not commensurable with either witness fraction.

The two exhaustive runs do give one population-level intersection.  The
twelve low-tail values first attained at order $14$ contain all $17$
non-ordinal-sum classes in the recorded tail: the $14$ classes at the nine
lower values require a balanced triple, while the three classes at the three
higher values have a bilateral pair certificate.  Each value is one complete
duality orbit, so classifying its representative and dual exhausts the
attainment set.

The following statements concern named retained witnesses and do not infer
population rates.
$L_{14,1,9}$, the worst-balanced non-sum poset of Theorem~\ref{thm:ladder},
has \emph{no} pair certificate and is resolved only by a balanced triple.
Every one of the $37$ recorded posets carrying a cycle of length $7$ or $8$
has balance $1/2$ and hence a half-balanced pair certificate.  In addition,
all $37$ are caught by the bilateral screen, which needs both orientations to
have at most one slave; this stronger fact is observed rather than forced.
And the families separate by width:
every retained low-tail witness has width $2$ or $3$ and every equality
representative width exactly $2$, against $7$ to $10$ for the long-cycle
witnesses.  Every equality-family class also has a bilateral certificate:
in a $T=\{a<b,c\}$ summand the ordered pair $(a,c)$ has the sole slave $b$,
while $(c,a)$ has none, and every other ordinal summand is comparable with
both elements.

\begin{observation}\label{obs:meet}
Every recorded order-$14$ poset carrying a majority cycle has balance constant
at least $\MINCYCBAL$, while the low tail is bounded above by $9/25$.
\end{observation}

The disjointness alone is weaker than the observed balance separation, and no
overlap probability can be inferred because the witnesses are retained
deterministically rather than sampled from the classes.  What
Observation~\ref{obs:meet} adds is that the separation is not marginal --- the
two families sit at opposite ends of the balance range, one pressed against
$1/3$ and the other against $1/2$.
Settling any of this over all classes would need a pass recording both
statistics together; the existing shards store marginals and selected
witnesses, so it cannot be done retrospectively.

\section{Verification}\label{sec:verify}

A witness is written \texttt{n:xxxx\ldots}, the order followed by $n$
four-digit hexadecimal upper-set masks: bit $y$ of mask $x$ is set exactly
when $x<y$ in the strict transitive closure.  The decoder validates
irreflexivity, antisymmetry, and transitive closure before use, so a malformed
or non-transitive string is rejected rather than silently interpreted.

We separate three kinds of check, because they carry very different weight.

\paragraph{External ground truth.}
Both runs' class totals agree exactly with the published number of unlabeled
posets on $14$ points \cite{heitzig-reinhold,brinkmann-mckay}.  A single
missing, duplicated, or truncated shard would shift a total by roughly
$10^{8}$, so digit-exact agreement is a strong check on enumeration
completeness and on the disjointness of the residue partition.  The census
cycle counts at orders $9$ through $13$ agree exactly with
\cite{deloof-census}, and the ladder values with Table 1 of
\cite{peczarski-ladders}.

Two further external checks are available and both pass.  First, Peczarski's
Table I \cite{peczarski-gold} supplies an external regression for the
certificate logic.  It gives, for $4\leq n\leq11$, the number of non-chain
ordinal-indecomposable
classes with dual pairs identified together with the number resolved by each
criterion.  Rebuilding that universe and classifying it with an independent
implementation of Lemma~\ref{lem:certificates} reproduces both his class
counts and his pair/triple split exactly at every order our independent
generator reaches, which is $4\leq n\leq8$.

At orders $4$ through $8$ he records respectively $6$, $21$, $111$, $725$,
and $6{,}474$ classes.  Their pair-certificate counts are $6$, $19$, $105$,
$704$, and $6{,}322$, while the balanced-triple counts are $0$, $2$, $6$,
$21$, and $152$; the independent classification returns every number exactly.

This is the only check that tests the certificate conditions against an
external source rather than against the enumeration, and it is the check the
order-$14$ decision most needs, since $\mathrm{open}=0$ is otherwise
self-reported.

Second, the census's count of ordinal-indecomposable classes is externally
determined.  Ordinal decomposition into indecomposable summands is unique, so
if $A(x)=\sum_n a(n)x^n$ with $a(0)=1$ and $a(n)$ the number of unlabeled
$n$-element posets, and $B(x)$ is the corresponding series for the
indecomposables, then $A=1/(1-B)$.  Inverting against the published totals
\cite{heitzig-reinhold,brinkmann-mckay} predicts $1{,}269{,}310{,}589{,}336$
indecomposable classes at order $14$ and $31{,}559{,}446{,}774$ at order $13$;
the census counts exactly those.  This validates the ordinal-sum decomposition
on which Theorem~\ref{thm:ladder} depends and which nothing else in the
pipeline tests.

The order-$14$ extremal data itself has, by construction, nothing external to
be checked against.

\paragraph{Independent recomputation.}
Every extremal witness is recomputed from its witness string alone by a
program sharing no algorithm with the census: where the census multiplies
prefix and suffix counts over the ideal lattice as in \eqref{eq:pairs}, the
checker counts, for each ordered pair separately, the linear extensions with
that pair's relation adjoined, by a plain forward recurrence.  At orders $5$
through $9$ the census is compared on every class against a brute-force
reference that enumerates linear extensions directly.

For the Gold Partition pass, a Python program generates every unlabeled poset
through order $7$ and applies the certificate definitions to explicitly
enumerated extensions, and a second implementation using a plain subset
recurrence is compared with the released classifier on every poset at orders
$8$ and $9$, including non-topological labelings.  Neither shares
extension-counting or certificate code with the classifier.  No independent
per-poset check is made above order $9$.

\paragraph{Internal consistency.}
For each run: the residues are disjoint and covering; every shard carries
exactly one summary and one parameter line; every shard records the same
binary hash; the census payloads match the SHA-256 sealed into their metadata
before publication; the census equality, tail, and cycle-budget counters report
no overflow; the certificate counters partition the total; and exactly one
class is a chain.  These checks would not detect a classifier that is
consistently wrong, which is why they are listed last.

\paragraph{Provenance.}
Each production binary hash is conditional on the recorded native build, GCC
13.3.0 on AMD EPYC 7B13 with \texttt{-march=native}: independent builds on two
hosts of that model are byte-identical, which does not generalise to other
silicon.  A portable build is provided and produces identical output, as do
GCC 13, 14, and 16 at orders $10$ and $11$.

\section{Open questions}\label{sec:open}

\begin{question}\label{q:relations}
Do $\Dig(P)$ and $\Dig_{\mathrm{inc}}(P)$ have the same cycle spectrum for
every finite poset?  Proposition~\ref{prop:three} settles length $3$;
Observation~\ref{obs:relations} is the evidence for the rest.
\end{question}

\begin{question}
Do posets whose majority digraph has bipartite cyclic part admit a
characterisation?  Observation~\ref{obs:even} gives a sufficient condition for
the absence of odd Condorcet cycles, and only a sufficient one: the converse
fails already at order $14$, where retained witnesses with cycle spectrum
exactly $\{4\}$ have a cyclic part that is not bipartite.  An undirected odd
cycle need not be consistently oriented, so it need not carry a directed one.
\end{question}

\begin{question}
Does Peczarski's broken-rung conjecture hold for all $n\geq9$?
\end{question}

\begin{question}
Is the family of Proposition~\ref{prop:equality} the complete equality locus
of the $1/3$--$2/3$ Conjecture at every order?
\end{question}

De Loof et al.\ also study cycle-free cut levels, which require repeated cycle
detection at many thresholds.  We did not compute them; it is the natural
follow-up to this census.

\section{Reproducibility}

The manuscript, source, tests, and compact aggregate data are available at
\[
  \text{\url{https://github.com/agupta/gold-partition-conjecture}}.
\]
The compact aggregate reports for both runs are included in the repository
under \texttt{data/}.  The complete order-$14$ shard archives, their checksum
manifests, and the build records are archived on Zenodo.  The concept DOI
\href{https://doi.org/10.5281/zenodo.21576029}{10.5281/zenodo.21576029}
resolves to the latest version; readers checking the data behind \emph{this}
paper should use the version DOI recorded in \texttt{data/census-n14.txt},
since the concept DOI will move when a later version is deposited.
The shard archive retains witness strings for every extremum and tail value,
every equality representative, and one witness per cycle length per shard.
The compact aggregate enriches its global extrema with canonical and Hasse
\texttt{digraph6} encodings.  Witness-specific claims can be recomputed from
those strings; global counts, completeness, and the equality classification
are checked by the aggregator over the complete sealed shard set.  The
retained-witness statistics are recomputed by
\texttt{scripts/analyze\_witness\_archive.py}.  The figures are generated
from fixed witness strings by \texttt{scripts/make\_census\_figures.py},
which re-verifies the asserted properties before drawing them.  Each archive
includes the source files used to build the corresponding production binary.

The calculations used \texttt{genposetg} 1.1 from nauty 2.9.1 and native GCC
13.3.0 builds on AMD EPYC 7B13 processors.  The repository build is portable
by default.  Tests of the shard protocol and aggregation rules are included
with the source, and the \texttt{make check} target runs the independent and
differential small-order checks, aggregation tests, and an end-to-end sharding
test.

An earlier version of this paper reported the Gold Partition result alone,
without the census of Sections~\ref{sec:balance}--\ref{sec:cycles}.

\paragraph{Acknowledgment of generative-AI assistance.}
Anthropic Claude Code (Claude 5 family) and OpenAI Codex (GPT-5.6 family) were
used extensively for software development and optimization, computational
experiment design, literature discovery, and drafting and editing the
manuscript.  The author selected the arguments and methods, checked the cited
sources and reported computations, and takes full responsibility for the
content.


\begin{thebibliography}{99}

\bibitem{aigner}
M.~Aigner,
\emph{A note on merging},
Order \textbf{2} (1985), 257--264.
\href{https://doi.org/10.1007/BF00333131}{doi:10.1007/BF00333131}.

\bibitem{brightwell}
G.~R. Brightwell,
\emph{Balanced pairs in partial orders},
Discrete Math. \textbf{201} (1999), 25--52.
\href{https://doi.org/10.1016/S0012-365X(98)00311-2}
{doi:10.1016/S0012-365X(98)00311-2}.

\bibitem{aires-kahn}
M.~Aires and J.~Kahn,
\emph{Balancing extensions in posets of large width},
\href{https://doi.org/10.48550/arXiv.2509.11549}
{arXiv:2509.11549} (2025).

\bibitem{aires-kahn-bounded}
M.~Aires and J.~Kahn,
\emph{Variance vs.\ range for linear extensions, and balancing extensions in
posets of bounded width},
\href{https://doi.org/10.48550/arXiv.2510.26134}
{arXiv:2510.26134} (2025).

\bibitem{bft}
G.~R. Brightwell, S.~Felsner, and W.~T. Trotter,
\emph{Balancing pairs and the cross product conjecture},
Order \textbf{12} (1995), 327--349.
\href{https://doi.org/10.1007/BF01110378}{doi:10.1007/BF01110378}.

\bibitem{brinkmann-mckay}
G.~Brinkmann and B.~D. McKay,
\emph{Posets on up to 16 points},
Order \textbf{19} (2002), 147--179.
\href{https://doi.org/10.1023/A:1016543307592}
{doi:10.1023/A:1016543307592}.

\bibitem{chan-pak}
S.~H. Chan and I.~Pak,
\emph{Linear extensions of finite posets},
EMS Surv. Math. Sci. (2025), published online first.
\href{https://doi.org/10.4171/EMSS/97}{doi:10.4171/EMSS/97}.

\bibitem{chen}
E.~Chen,
\emph{A family of partially ordered sets with small balance constant},
Electron. J. Combin. \textbf{25} (2018), \#P4.43.
\href{https://doi.org/10.37236/7337}{doi:10.37236/7337}.

\bibitem{deloof-ideals}
K.~De Loof, H.~De Meyer, and B.~De Baets,
\emph{Exploiting the lattice of ideals representation of a poset},
Fund. Inform. \textbf{71} (2006), 309--321.
\href{https://doi.org/10.3233/FUN-2006-712-309}
{doi:10.3233/FUN-2006-712-309}.

\bibitem{deloof-census}
K.~De Loof, B.~De Baets, and H.~De Meyer,
\emph{Counting linear extension majority cycles in partially ordered sets
on up to 13 elements},
Comput. Math. Appl. \textbf{59} (2010), 1541--1547.
\href{https://doi.org/10.1016/j.camwa.2009.12.021}
{doi:10.1016/j.camwa.2009.12.021}.

\bibitem{dolores-cuenca}
E.~R. Dolores-Cuenca, A.~Guzm\'an-S\'aenz, and S.~Kim,
\emph{The Gold Partition Conjecture and the lexicographic sum of posets},
\href{https://doi.org/10.48550/arXiv.2410.12494}
{arXiv:2410.12494} (2024).

\bibitem{fredman}
M.~L. Fredman,
\emph{How good is the information theory bound in sorting?},
Theoret. Comput. Sci. \textbf{1} (1976), 355--361.

\bibitem{fishburn-cycles}
P.~C. Fishburn,
\emph{On the family of linear extensions of a partial order},
J. Combin. Theory Ser. B \textbf{17} (1974), 240--243.
\href{https://doi.org/10.1016/0095-8956(74)90030-6}
{doi:10.1016/0095-8956(74)90030-6}.

\bibitem{fishburn-lem}
P.~C. Fishburn,
\emph{On linear extension majority graphs of partial orders},
J. Combin. Theory Ser. B \textbf{21} (1976), 65--70.
\href{https://doi.org/10.1016/0095-8956(76)90028-9}
{doi:10.1016/0095-8956(76)90028-9}.

\bibitem{gehrlein-fishburn}
W.~V. Gehrlein and P.~C. Fishburn,
\emph{Linear extension majority cycles for small ($n\leq9$) partial orders},
Comput. Math. Appl. \textbf{20} (1990), 41--44.
\href{https://doi.org/10.1016/0898-1221(90)90239-G}
{doi:10.1016/0898-1221(90)90239-G}.

\bibitem{heitzig-reinhold}
J.~Heitzig and J.~Reinhold,
\emph{The number of unlabeled orders on fourteen elements},
Order \textbf{17} (2000), 333--341.
\href{https://doi.org/10.1023/A:1006431609027}
{doi:10.1023/A:1006431609027}.

\bibitem{kahn-saks}
J.~Kahn and M.~Saks,
\emph{Balancing poset extensions},
Order \textbf{1} (1984), 113--126.
\href{https://doi.org/10.1007/BF00565647}{doi:10.1007/BF00565647}.

\bibitem{kislitsyn}
S.~S. Kislitsyn,
\emph{Finite partially ordered sets and their associated sets of
permutations},
Mat. Zametki \textbf{4} (1968), 511--518.

\bibitem{linial}
N.~Linial,
\emph{The information-theoretic bound is good for merging},
SIAM J. Comput. \textbf{13} (1984), 795--801.

\bibitem{mckay-piperno}
B.~D. McKay and A.~Piperno,
\emph{Practical graph isomorphism, II},
J. Symbolic Comput. \textbf{60} (2014), 94--112.
\href{https://doi.org/10.1016/j.jsc.2013.09.003}
{doi:10.1016/j.jsc.2013.09.003}.

\bibitem{olson-sagan}
E.~J. Olson and B.~E. Sagan,
\emph{On the $1/3$--$2/3$ Conjecture},
Order \textbf{35} (2018), 581--596.
\href{https://doi.org/10.1007/s11083-017-9450-3}
{doi:10.1007/s11083-017-9450-3}.

\bibitem{peczarski-gold}
M.~Peczarski,
\emph{The Gold Partition Conjecture},
Order \textbf{23} (2006), 89--95.
\href{https://doi.org/10.1007/s11083-006-9033-1}
{doi:10.1007/s11083-006-9033-1}.

\bibitem{peczarski-thin}
M.~Peczarski,
\emph{The Gold Partition Conjecture for 6-thin posets},
Order \textbf{25} (2008), 91--103.
\href{https://doi.org/10.1007/s11083-008-9081-9}
{doi:10.1007/s11083-008-9081-9}.

\bibitem{peczarski-comments}
M.~Peczarski,
\emph{Comments on the Golden Partition Conjecture},
Contrib. Discrete Math. \textbf{12} (2017), 106--109.
\href{https://doi.org/10.55016/ojs/cdm.v12i1.62349}
{doi:10.55016/ojs/cdm.v12i1.62349}.

\bibitem{peczarski-ladders}
M.~Peczarski,
\emph{The worst balanced partially ordered sets---ladders with broken rungs},
Exp. Math. \textbf{28} (2019), 181--184.
\href{https://doi.org/10.1080/10586458.2017.1368050}
{doi:10.1080/10586458.2017.1368050}.

\bibitem{sah}
A.~Sah,
\emph{Improving the $1/3$--$2/3$ conjecture for width two posets},
Combinatorica \textbf{41} (2021), 99--126.
\href{https://doi.org/10.1007/s00493-020-4091-3}
{doi:10.1007/s00493-020-4091-3}.

\end{thebibliography}
\end{document}